\documentclass[final,a4paper]{amsart}

\usepackage[english]{babel}
\usepackage[utf8]{inputenc}

\usepackage[T1]{fontenc}
\usepackage{lmodern}

\usepackage{enumitem}
\usepackage[notref,notcite,color]{showkeys}
\usepackage{verbatim}
\usepackage{url}

\def\clap#1{\hbox to 0pt{\hss#1\hss}}

\def\mathrlap{\mathpalette\mathrlapinternal}

\def\mathrlapinternal#1#2{%
\rlap{$\mathsurround=0pt#1{#2}$}}

\usepackage{amsmath,amssymb,ifthen,xspace}

\newcommand{\lpfont}{\mathrm}

\newcommand*{\lp}[2][]{\ensuremath{
    \lpfont{
      \ifthenelse{\equal{#1}{}}{#2}{#1 \text{-} #2}}}\xspace}
\newcommand*{\lpp}[3][]{\ensuremath{
    \lpfont{
      \ifthenelse{\equal{#1}{}}{#2}{#1 \text{-} #2}}(#3)}\xspace}
\newcommand*{\lph}[3][]{\ensuremath{
    \lpfont{
      \ifthenelse{\equal{#1}{}}{{#2}^{#3}}{#1 \text{-} {#2}^{#3}}}}\xspace
}

\newcommand*{\lpf}[1]{(\lp{#1})}
\newcommand*{\Weak}{{\upharpoonright}}

\newcommand*{\lf}[1]{\lpfont{#1}}

\newcommand*{\ls}[2][]{\ensuremath{
  \lpfont{
      \ifthenelse{\equal{#1}{}}{#2}{#1 \text{-} #2}}}\xspace}

\newcommand*{\WIDEomega}[1]{\vphantom{\lpfont{\widehat{#1}}}\smash[t]{\lpfont{\widehat{#1}}}^\omega}

\newcommand*{\EPAw}{\ls{\WIDEomega{E\text{-}PA}\Weak}}
\newcommand*{\WEPAw}{\ls{\WIDEomega{WE\text{-}PA}\Weak}}

\newcommand{\IMPL}[1][]{
  \ifthenelse{\equal{#1}{}}{\mathop{\rightarrow}}{\mathop{\stackrel{#1}{\longrightarrow}}}}

\newcommand*{\AND}{\mathrel{\land}}
\newcommand*{\OR}{\mathrel{\lor}}

\newcommand*{\IFF}[1][]{
  \ifthenelse{\equal{#1}{}}{\mathrel{\leftrightarrow}}{\mathrel{\stackrel{#1}{\longleftrightarrow}}}}

\newcommand*{\Quantor}[2]{{#1 #2}\,}
\newcommand*{\Forall}[1]{\Quantor{\forall}{#1}}
\newcommand*{\Exists}[1]{\Quantor{\exists}{#1}}
\newcommand*{\ExistsUn}[1]{\Quantor{\exists!}{#1}}

\newcommand*{\Nat}{\ensuremath{\mathbb{N}}}

\newcommand*{\tup}[2][]{
  \ifthenelse{\equal{#1}{}}{\ushort{#2}}{\ushort{#2}^{\ushort{#1}}}}

\DeclareMathOperator{\lth}{lth}

\newcommand{\qf}{\textit{\!qf}}

\newcommand{\sizeMid}{\;\middle\vert\;}

\newcommand*{\Theorem}{Theorem}
\newcommand*{\Proposition}{Proposition}
\newcommand*{\Lemma}{Lemma}
\newcommand*{\Corollary}{Corollary}
\newcommand*{\Definition}{Definition}
\newcommand*{\Remark}{Remark}
\newcommand*{\Notation}{Notation}

\theoremstyle{plain}
\newtheorem{theorem}{\Theorem}
\newtheorem{proposition}[theorem]{\Proposition}

\newtheorem{lemma}[theorem]{\Lemma}
\theoremstyle{definition}
\newtheorem{definition}[theorem]{\Definition}
\theoremstyle{remark}
\newtheorem{remark}[theorem]{\Remark}

\newcommand{\U}{\mathcal{U}}
\newcommand{\F}{\mathcal{F}}
\newcommand{\ps}[1]{\mathcal{P}(#1)}

\newcommand{\T}{\mathbf{T}}

\title[Ultrafilters, program extraction and reverse mathematics]{Non-principal ultrafilters, program extraction and higher order reverse mathematics}

\author{Alexander P.\ Kreuzer}
\address{Fachbereich Mathematik, Technische Universit\"{a}t Darmstadt\\
Schlossgartenstra{\ss}e~7, 64289 Darmstadt, Germany
}
\email{akreuzer@mathematik.tu-darmstadt.de}
\urladdr{http://www.mathematik.tu-darmstadt.de/~akreuzer}
\thanks{The author is supported by the German Science Foundation (DFG Project KO 1737/5-1).}
\thanks{I am grateful to Ulrich Kohlenbach for useful
discussions and suggestions for improving the presentation
of the material in this article.}

\date{\today}

\begin{document}
\begin{abstract}
  We investigate the strength of the existence of a non-principal ultrafilter over fragments of higher order arithmetic. 

  Let \lpf{\U} be the statement that a non-principal ultrafilter exists and let \ls{ACA_0^\omega} be the higher order extension of \ls{ACA_0}.
  We show that $\ls{ACA_0^\omega}+ \lpf{\U}$ is $\Pi^1_2$\nobreakdash-\hspace{0pt}conservative over \ls{ACA_0^\omega} and thus that $\ls{ACA_0^\omega} +\lpf{\U}$ is conservative over \ls{PA}.

  Moreover, we provide a program extraction method and show that from a proof of a strictly $\Pi^1_2$ statement $\Forall{f}\Exists{g} \lf{A_\qf}(f,g)$ in $\ls{ACA_0^\omega}+\lpf{\U}$ a realizing term in G\"odel's system $T$ can be extracted. This means that one can extract a term $t\in T$, such that $\Forall{f} A_\qf(f,t(f))$.
\end{abstract}

\maketitle

In this paper we will investigate  the strength of the existence of a non-principal ultrafilter over fragments of higher order arithmetic.
We will classify the consequences of this statement in the spirit of reverse mathematics. Furthermore, we will provide a program extraction method.

Let \lpf{\U} be the statement that a non-principal ultrafilter on $\Nat$ exists. Let \ls{RCA_0^\omega}, \ls{ACA_0^\omega} be the extensions of \ls{RCA_0} resp.\ \ls{ACA_0} to higher order arithmetic as introduced by Kohlenbach in \cite{uK05b}. In \ls{RCA_0^\omega} or \ls{ACA_0^\omega} the statement \lpf{\U} can be formalized using an object of type $\Nat^\Nat \longrightarrow \Nat$.

Further, let Feferman's $\mu$ be a functional of type $\Nat^\Nat\longrightarrow \Nat$ satisfying 
\[
f(\mu(f))= 0\quad\text{if}\quad \Exists{x} f(x)=0 
\]
and let \lpf{\mu} be the statement that such a functional exists. It is clear that \lpf{\mu} implies arithmetical comprehension.

 We will show that
\begin{itemize}
\item over \ls{RCA_0^\omega} the statement \lpf{\U} implies \lpf{\mu} and therefore also \ls{ACA_0^\omega}, and that
\item $\ls{ACA_0^\omega}+\lpf{\mu}+\lpf{\U}$ is $\Pi^1_2$-conservative over \ls{ACA_0^\omega} and therefore also conservative over \ls{PA}. Moreover, we will show that from a proof of $\Forall{f}\Exists{g} \lf{A_\qf}(f,g)$ in $\ls{ACA_0^\omega} + \lpf{\mu}+ \lpf{\U}$, where $\lf{A_\qf}$ is quantifier free, one can extract a realizing term $t$ in G\"odel's system~$T$, i.e.\ a term such that $\Forall{f}\lf{A_\qf}(f,t(f))$. 
\end{itemize}

The system $\ls{ACA_0^\omega}+\lpf{\mu}+\lpf{\U}$ is strong, one can carry out nearly all ultralimit and non-standard arguments.
For instance one can carried out in this theory the construction of Banach limits and many Loeb measure constructions.
Our results show that this system is weak with respect to $\Pi^1_2$ sentences.
Moreover, our program extraction result show that one can still obtain constructive (even primitive recursive in the sense of G\"odel) realizers and bounds from proofs using highly non-constructive objects like non-principal ultrafilter.

Using this technique it is possible to extract bounds from proofs using ultralimits and non-standard technique. Such proofs do occur in mathematics, for instance in metric fixed point theory, see \cite{AK90} and \cite{KS01}. In \cite{pG06b} Gerhardy extracted a rate of proximity of such a proof by eliminating the ultrafilter by hand. Our result here show that this can be done with any such argument.


\subsubsection*{Comparison to other approaches}
Solovay first used partial ultrafilter. He constructed a filter which acts on the hyperarithemtical sets like a non-principal ultrafilter. With this he show an effective version of the Galvin-Prikry theorem, see \cite{rS78}. 
His construction of the partial ultrafilter is similar to ours. Avigad analyzed his result in terms of reverse mathematics and formalized this particular proof in \ls{ATR_0}, see \cite{jA98}.
However, this result does not follow from our meta-theorem, since it not only uses a non-principal ultrafilter but also substantial amounts of transfinite recursion. 

Using our approach one also obtains upper bounds on the strength of non-standard analysis and program extraction methods. This can be done by constructing a ultrapower model of non-standard analysis in $\ls{ACA_0^\omega}+\lpf{\mu}+ \lpf{\U}$. If one is not interested in the ultrafilter but only in the axiomatic treatment of non-standard analysis one can obtain refined results by interpreting it directly, see for instance \cite{jA05}, \cite{jK06} and for program extraction \cite{BBS}.

Palmgren used in \cite{eP99} an approach similar to ours to interpret non-standard arithmetic. He builds (partial) non-principal ultrafilters for the definable sets of a fixed level in the arithmetic hierarchy. He obtains conservations result very similar to ours. However he cannot treat ultrafilter nor obtains program extraction.

In reverse mathematics idempotent ultrafilters are considered in the context of Hindman's theorem, which can be proven using an idempotent ultrafilter (or at least a countable part of it), see Hirst \cite{jH04} and Towsner \cite{hT11}. 
We code ultrafilter over countable fields like Hirst does. However, our construction of ultrafilters is different since we are not aiming for idempotent ultrafilters.
An idempotent ultrafilter is a very special ultrafilter and it seems that even the construction of countable parts of an idempotent ultrafilter requires a system that is proof theoretically stronger than $\ls{ACA_0^\omega} + \lpf{\mu}$ and is therefore beyond our method.

\subsection*{Logical system}

We will work in fragments of Peano arithmetic in all finite types.
The set of all finite types $\T$ is defined to be the smallest set that satisfies 
\[
0\in \T, \qquad \rho,\tau\in \T \Rightarrow \tau(\rho)\in \T
.\]
The type $0$ denotes the type of natural numbers and the type $\tau(\rho)$ denotes the type of functions from $\rho$ to $\tau$. The type $0(0)$ is abbreviated by $1$ the type $0(0(0))$ by $2$. The degree of a type is defined by
\[
deg(0):=0 \qquad deg(\tau(\rho)) := \max(deg(\tau),deg(\rho)+1)
.\]
The type of a variable will sometimes be written as superscript of a term or as subscript of an equality sign.

The system \ls{RCA_0^\omega} is the extension of \ls{RCA_0} to all finite types. The systems \ls{WKL_0^\omega}, \ls{ACA_0^\omega} are defined to be $\ls{RCA_0^\omega}+\lp{WKL}$ resp.\ $\ls{RCA_0^\omega}+\ls[\Pi^0_1]{CA}$. All of these system are conservative over their second order counterpart via the embedding of sets as characteristic functions. For details see \cite{uK05b}.

Let \lp[QF]{AC^{1,0}} be the schema
\[
\Forall{f^1}\Exists{x^0}\lf{A_\qf}(f,x) \IMPL \Exists{F^2}\Forall{f^1} \lf{A_\qf}(f,F(f))
.\]
All of the above defined systems include \lp[QF]{AC^{1,0}}. This schema is the higher order equivalent to recursive comprehension (\lp[\Delta^0_1]{CA}).

The terms of \ls{RCA_0^\omega} consist of $0^0$, the successor function $S^1$, lambda combinators $\Pi$ and $\Sigma$ for all types, which provide lambda abstraction, and the recursor $R_0$.
The recursor $R_0$ satisfies the following equations
\[
R_00yz =_0 0, \qquad R_0(x+1)yz =_0 z(Rxyz)x 
.\]
It provides primitive recursion (in the sense of Kleene). The closed terms of \ls{RCA_0^\omega} are also called $T_0$ (for the restriction of G\"odel's system $T$ to recursion of type $0$).
If one adds (impredicative) recursors $R_\rho$ for all types $\rho\in \T$ to $T_0$ one obtains the full system $T$ of G\"odel. The functions in $T$ are called primitive recursive in the sense of G\"odel.
By $T_0[F]$ we will denote the system resulting from adding a function(al) $F$ to $T_0$.

The system \ls{RCA_0^\omega} has a functional interpretation (always combined with elimination of extensionality and a negative translation) in $T_0$. The system \ls{ACA_0^\omega} has a functional interpretation in $T_0[\mu]$ if one interprets comprehension using $\mu$ or in $T_0[B_{0,1}]$ if one interprets comprehension using the bar recursor of lowest type $B_{0,1}$. See \cite{uK05b} and \cite{AF98} for the interpretation using $\mu$ and \cite[Section~11]{uK08} for the interpretation using $B_{0,1}$. For a general survey on the functional interpretation see \cite{uK08} and \cite{AF98}.

\begin{definition}[non-principal ultrafilter, \lpf{\U}]\label{def:ultra}
  Let \lpf{\U} be the statement that there exists a non-principal ultrafilter (on \Nat):
  \[
  \lpf{\U}\colon\left\{
  \begin{aligned}
    \Exists{\U^2} \big(\ &\Forall{X} \left(X\in\U \OR \overline{X}\in\U\right) \\
    \AND\, &\Forall{X^1,Y^1} \left(X \cap Y \in \U \IMPL Y\in \U\right) \\
    \AND\, &\Forall{X^1,Y^1} \left(X,Y\in \U \IMPL (X\cap Y)\in\U\right)  \\
    \AND\, &\Forall{X^1} \left(X\in\U \IMPL \Forall{n}\Exists{k>n} (k\in X)\right) \\
    \AND\, &\Forall{X^1} \left(\U(X) =_0 \U(\lambda n. \min(X(n),1))\right)\big)
  \end{aligned}
  \right.
  \]
  Here $X\in \U$ is an abbreviation for $\U(X)=0$. The type $1$ variables $X,Y$ are viewed as characteristic function of sets, where $n\in X$ is defined to be $X(n) = 0$.
  The operation $\cap$ is defined as taking the pointwise maximum of the characteristic functions. With this the intersection of two sets can be expressed in a quantifier-free way.
  The last line of the definition states that $\U$ yields the same value for different characteristic functions of the same set.

  For notational ease we will usually add a Skolem constant $\U$ and denote this also with \lpf{\U}.

  The second line in the definition of \lpf{\U} is equivalent to the following axiom usually found in the axiomatization of (ultra)filters:
  \[
  \Forall{X,Y} \left( X\subseteq Y \AND X\in\U \IMPL Y\in\U\right)
  .\]
  We avoided this statement in $\lpf{\U}$ since $\subseteq$ cannot be expressed in a quantifier free way.
\end{definition}

\begin{lemma}[finite partition property]\label{lem:finitepart}
  The ultrafilter $\U$ satisfies the finite partition property over \ls{RCA_0^\omega}.

  This means that for each finite partition $(X_i)_{i<n}$ of $\Nat$ the following holds
  \[
  \ls{RCA_0^\omega} +  \lpf{\U} \vdash \ExistsUn{i<n} X_i\in \U
  .\]
\end{lemma}
\begin{proof}
  We prove by quantifier-free induction on $m$ the statement 
  \begin{equation}\label{eq:finpartind}
  \ExistsUn{i\le m} \bigg(\Big( i < m \IMPL X_i \in \U\Big) \AND\ \Big( i = m \IMPL \bigcup_{\mathrlap{j=m}}^{\mathrlap{n-1}}  X_j \in \U \Big)\bigg)
  .\end{equation}
  In the cases $m\le 2$ the statement follows directly from \lpf{\U}.
  For the induction step we assume that the statement for $m$ holds. This means there exists an $i$ as stated in \eqref{eq:finpartind}. If $i<m$ then this $i$ also satisfies \eqref{eq:finpartind} with $m$ replaced by $m+1$ and we are done. Otherwise we have $\bigcup_{j=m}^{n-1}  X_j\in \U$.

  The axiom \lpf{\U} yields
  \[
  \bigcup_{\mathrlap{j=0}}^{m}  X_j\in \U\ \OR\ \bigcup_{\mathrlap{j=m+1}}^{\mathrlap{n-1}} X_j\in \U
  .\]
  If the left side of the disjunction holds then
  \[
  X_m = \bigcup_{\mathrlap{j=0}}^{m}  X_j \cap \bigcup_{\mathrlap{j=m}}^{\mathrlap{n-1}}  X_j  \in \U
  \]
  and $i:=m$ satisfies the \eqref{eq:finpartind} with $m$ replaced by $m+1$. If the right side of the disjunction holds $i:=m+1$ satisfies \eqref{eq:finpartind}.

  The lemma follows from \eqref{eq:finpartind} by taking $m:= n$.
\end{proof}

\begin{theorem}\label{thm:ultraca}
  \[
  \ls{RCA_0^\omega} + \lpf{\U} \vdash \lpf{\mu}
  \]
  In particular $\ls{RCA_0^\omega} + \lpf{\U}\vdash \ls{ACA_0^\omega}$.
\end{theorem}
\begin{proof}
  Let $f\colon \Nat \to \Nat$ be a function.
  The set $X_f := \{ x\in\Nat \mid \Exists{x'<x} f(x')=0 \}$ is cofinal if $\Exists{x} f(x) = 0$,  if not then the set $X_f$ is empty. Hence
  \[
  X_f\in \U \quad\text{if{f}}\quad \Exists{x} f(x) = 0
  .\]
  From this it follows that
  \[
  \Forall{f} \Exists{x} \left( X_f\in \U \IMPL f(x) = 0\right)
  .
  \]
  An application of \lp[QF]{AC^{1,0}} now yields a functional satisfying \lpf{\mu}.
\end{proof}

\begin{theorem}[Program extraction]\label{thm:progex}
  Let $\lf{A_\qf}(f,g)$ be a quantifier free formula of \ls{RCA_0^\omega}  containing only $f,g$ free. In particular $\lf{A_\qf}$ must not contain $\mu$ or $\U$.

  If 
  \[
  \ls{ACA_0^\omega} + \lpf{\mu} +  \lpf{\U} \vdash \Forall{f^1}\Exists{g^1} \lf{A_\qf}(f,g)
  \]
  then one can extract a closed term $t\in T$ such that 
  \[
  \Forall{f} \lf{A_\qf}(f,tf)
  .\]
\end{theorem}

\noindent
The proof of this theorem proceeds in five steps:
\begin{enumerate}[label=\arabic*.]
\item Using the functional interpretation and proof theoretic methods developed in \cite{swkl} we show that a proof of the statement  
  \[
  \ls{ACA_0^\omega} + \lpf{\mu} + \lpf{\U} \vdash \Forall{f}\Exists{g} \lf{A_\qf}(f,g)
  \]
  can be normalized in such a way that each application of the functional $\U$ that occurs in the proof has the form $\U(t[n^0])$, where $t$ is a term that contains only $n$ free and with $\lambda n . t \in T_0[\U]$. (We do not have to consider $\mu$ here, since it can be defined from $\U$ by Theorem~\ref{thm:ultraca}.)
  In particular this shows the ultrafilter $\U$ is used only on countable many sets.
\item We show that we can construct in $\ls{RCA_0^\omega}+\lpf{\mu}$ a \emph{partial ultrafilter}, that is an object that behaves like an ultrafilter on the sets that occur in the proof. We then replace $\U$ by this partial ultrafilter and obtain a proof of $\Forall{f}\Exists{g} \lf{A_\qf}(f,g)$ in $\ls{RCA_0^\omega}+\lpf{\mu}$.
\item
  The theory $\ls{RCA_0^\omega}+\lpf{\mu}$ is conservative over \ls{ACA_0^\omega}, see \cite{AF98}, hence we obtain a proof in this theory.
\item Applying the functional interpretation to this statement and interpreting the comprehension using $B_{0,1}$ yields a term $t^2\in T_0[B_{0,1}]$, such that
  \[
  \Forall{f} A_\qf(f,tf)
  .\]
\item Since this term  $t$ is only of type $2$, one can use an ordinal analysis of the bar recursor to eliminated it and obtain a new term $t'\in T$, such that $t'=_2 t$ and hence that
  \[
  \Forall{f} A_\qf(f,t'f)
  .\]
\end{enumerate}

Before we prove this theorem we show how to construct a partial ultrafilter and provide some proof theoretic lemmata.

\subsection*{Partial ultrafilter}

\begin{definition}[partial ultrafilter]\mbox{}
  \begin{itemize}
  \item Call a set $\mathcal{A}\subseteq\ps\Nat$ of subsets of natural numbers, that is closed under complement, finite unions and finite intersections, an \emph{algebra}.
  \item Let $\mathcal{A}$ be an algebra. Call a set $\F\subseteq \mathcal{A}$ a \emph{partial non-principal ultrafilter} for $\mathcal{A}$ if{f} $\F$ satisfies the non-principal ultrafilter axioms in Definition~\ref{def:ultra} relativized to $\mathcal{A}$, i.e.
    \[
    \left\{
    \begin{aligned}
       &\Forall{X\in \mathcal{A}} \left(X\in\F \OR \overline{X}\in\F\right) \\
      \AND\, &\Forall{X,Y\in \mathcal{A}} \left(X \cap Y \in \F \IMPL Y\in \F\right) \\
      \AND\, &\Forall{X,Y\in \mathcal{A}} \left(X,Y\in \F \IMPL (X\cap Y)\in\F\right)  \\
      \AND\, &\Forall{X\in\mathcal{A}} \left(X\in\F \IMPL \Forall{n}\Exists{k>n} k\in X\right) \\
      \AND\, &\Forall{X^1} \left(\F(X) =_0 \F(\lambda n . \min(X(n),1))\right).
    \end{aligned}
    \right.
    \]   
  \end{itemize}
\end{definition}
It is easy to see that one can extend in \ls{RCA_0^\omega} every sequence of sets to a countable algebra.
One should also note that partial non-principal ultrafilters for countable algebras are also countable.
A partial ultrafilter $\F$ can be viewed as the closed subset $\{\U\in \beta\Nat \mid \U \supseteq \F\}$ of the Stone-\v{C}ech compactification $\beta \Nat$.

\begin{proposition}\label{pro:ultraex}
  Let $\mathcal{A}$ be a countable algebra and let $\F=(F_i)_{i\in\Nat}$ be a countable partial non-principal ultrafilter for $\mathcal{A}$.
  Then $\ls{RCA_0^\omega}+\lpf{\mu}$ proves that for each countable extension $\mathcal{\tilde{A}}= (\tilde{A}_i)_{i\in \Nat}\supseteq \mathcal{A}$ there exists a partial non-principal ultrafilter $\tilde{\F}\supseteq \F$.
\end{proposition}

\begin{proof}
  In the following let $x$ be the code for a tuple $\langle x_0,\dots, x_{\lth(x) - 1}\rangle$ in $2^{<\Nat}$.
  Let
  \[
  \tilde{A}^x := \bigcap_{i<\lth x}
  \begin{cases}
    \tilde{A}_i & \text{if $x_i = 0$,} \\
    \overline{\tilde{A}_i} & \text{if $x_i = 1$.}
  \end{cases}
  \]
  Using quantifier free induction one easily sees that for every $n$ the set $\left\{ \tilde{A}^x \sizeMid x \in 2^n \right\}$ defines a partition of $\Nat$, i.e.
  \begin{equation}\label{eq:part}
    \Forall{n}\ExistsUn{x \in 2^n} \left(z \in \tilde{A}^{x}\right) \quad \text{for all $z$}
  .\end{equation}

  Define a $\Pi^0_2$\nobreakdash-0/1-tree $T$ by
  \[
  T(x) \quad \text{if{f}} \quad \Forall{j} \left(\text{$\tilde{A}^x \cap F_j$ is infinite}\right)
  .\]
  The tree $T$ is infinite because otherwise we would have
  \[
  \Exists{n} \Forall{x\in 2^n} \Exists{j}\Exists{y} \Forall{z>y} \ z\notin  \tilde{A}^x\cap F_j 
  .\]
  The bounded collection principle \lp[\Pi^0_1]{CP} yields
  \begin{equation}\label{eq:treefin}
  \Exists{n} \Exists{j^*,y^*} \Forall{x\in 2^n} \Forall{z>y^*} z\notin  \tilde{A}^x\cap \bigcap_{\mathrlap{j\le j^*}} F_j
  .\end{equation}
  The set $\bigcap_{j\le j^*} F_j$ is in $\F$ and is therefore infinite. In particular it contains an element $z$ which is bigger than $y^*$. Because $\tilde{A}^x$ with $x\in 2^n$ defines a partition of $\Nat$ there is an $x$ such that $z\in \tilde{A}^x$. This contradicts \eqref{eq:treefin} and therefore the tree $T$ is infinite.

  Hence we obtain using \lp[\Pi^0_2]{WKL} (which is provable in \ls{ACA_0^\omega} and hence using $\mu$) an infinite branch $b$ of $T$.
  The set 
  \[
  \mathcal{\tilde{F}} := \F \cup \left\{ \tilde{A}_i \sizeMid b(i) =0 \right\}
  \]
  defines then a partial non-principal ultrafilter for $\mathcal{\tilde{A}}$.
  The characteristic function of $\mathcal{\tilde{F}}$ is given by
  \[
  \chi_{\mathcal{\tilde{F}}}(B) :=
  \begin{cases}
    0 & \text{if\, $\left( B\in \F\right) \OR \Exists{i} \left(b(i) =_0 0 \AND A_i=B\right)$,} \\
    1 & \text{otherwise.}
  \end{cases}
  \]
  The set equality ($A_i=B$) can be defined using $\mu$, therefore $\mathcal{\tilde{F}}$ is definable.
\end{proof}

\subsection*{Proof theory}

The system \ls{RCA_0^\omega} contains full extensionality. This means roughly that for 
a functional $\Phi$ and functions $f,g$ one has $\Phi(f)=_0 \Phi(g)$ if $f$ and $g$ are extensionally equal (i.e.\ $\Forall{x} f(x) =_0 g(x)$). Extensionality cannot be expressed in a purely universal statement and therefore contains some constructive content.
For this reason the functional interpretation cannot handle this general form of extensionality directly and it has to be eliminated beforehand.
The system \ls{RCA_0^\omega} is formulated in a way that this can be done using standard methods, i.e.\ the elimination of extensionality, see for instance \cite[Section~10.4]{uK08}. Since we added a new higher order constant $\U$ we have to check manually that this constant is extensional.
This will be done in the following lemma.
To formulate it we will need a \emph{weakly extensional system}, i.e.\ a system in which extensionality is restricted to a rule of extensionality that only allows quantifier free premises. We will use $\WEPAw + \lp[QF]{AC^{1,0}}$. This system is the weakly extensional counterpart to \ls{RCA_0^\omega} in the sense that \ls{RCA_0^\omega} results from $\WEPAw + \lp[QF]{AC^{1,0}}$ by adding the extensionality axioms. (In other words $\ls{RCA_0^\omega} \equiv \EPAw + \lp[QF]{AC^{1,0}}$.)

\begin{lemma}[Elimination of extensionality]\label{lem:ultraexelim}
  The system $\WEPAw + \lpf{\U}$ proves that $\U$ is extensional, i.e.
  \[
  \Forall{X,Y} \big(\Forall{k}\left(k\in X \IFF k\in Y\right) \IMPL \left(X\in \U \IFF Y\in \U\right)\big)
  .\]
  
  In particular, the elimination of extensionality is applicable to $\ls{RCA_0^\omega}+\lpf{\U}$. This means the following rule holds:
  If $\lf{A}$ is a statement that contains only quantification over variables of degree $\le 1$ and 
  \[
  \ls{RCA_0^\omega} \vdash \lpf{\U} \IMPL \lf{A}
  \]
  then
  \[
  \WEPAw + \lp[QF]{AC^{1,0}} \vdash \lpf{\U} \IMPL \lf{A}
  .\]
\end{lemma}
\begin{proof}
  Suppose that $\U$ is not extensional. Then there exist two sets $X,Y$, such that
  \[
  \Forall{k} \left(k\in X \IFF k\in Y)\right)
  \quad\text{and}\quad X\in \U \AND Y\notin \U.
  \]
  By the axiom \lpf{\U} we obtain that $\overline{Y}\in \U$ and with this 
  \[ 
  X \cap \overline{Y} \in \U
  .\]
  By the last line of \lpf{\U} there exists an $n\in X\cap\overline{Y}$. This contradicts the assumption and we conclude that $\U$ is extensional.

  For the elimination of extensionality we use the techniques presented in Section~10.4 of \cite{uK08}. We will also use the notation introduced in this section for the rest of this proof.

  The extensionality of $\U$ translates into $\U =^e \U$. Since \lpf{\U} is (after the Skolemization) analytic and the constant $\U$ is extensional, we obtain $\lpf{\U}_e \IFF \lpf{\U}$.
  Because $\lf{A}$ does not contain quantification of degree $> 1$ we also obtain that $\lf{A}_e$ is equivalent to $\lf{A}$.
  Hence $\lpf{\U} \IMPL \lf{A}$ does not change under the $(\cdot)_e$ relativization.

  The lemma follows now from Proposition~10.45 in \cite{uK08} relativized according to \cite[Section~10.5]{uK08} to \lp{RCA_0^\omega}.
\end{proof}

The next theorem will provide the term normalization that is need for the proof of Theorem~\ref{thm:progex}.
\begin{theorem}[term-normalization for degree $2$]\label{thm:cut2}
  Let $F_1,\dots,F_n$ be constants of degree~$\le 2$.

  For every term $t^1\in T_0[F_1,\dots,F_{n}]$ there is a term $\tilde{t}\in T_0[F_0,\dots,F_{n-1}]$ with
  \[
  \WEPAw \vdash t =_1 \tilde{t}
  \]
  and such that
  every occurrence of an $F_i$ in $\tilde{t}$ is of the form
  \[
  F_i(\tilde{t}_0[y^0], \dots, \tilde{t}_{k-1}[y^0])
  .\]
  Here $k$ is the arity of $F_i$, and $\tilde{t}_j[y^0]$ are fixed terms whose only free variable is~$y^0$.
\end{theorem}
\begin{proof}
  See Theorem~20 in \cite{swkl}. For a reference see also \cite[proof of proposition 4.2]{uK99}. This normalization is similar to the normalization described in Section~8.3 of \cite{AF98}.
\end{proof}

The axiom \lpf{\U} can be prenext to a statement of the form
\[
\Exists{\U^2}\Forall{X^1,Y^1}\Forall{n}\Exists{k} 
\begin{aligned}[t]
  \big(\hphantom{\AND\,}&\left(X\in\U \OR \overline{X}\in\U\right) \\
  \AND\, &\left(X \cap Y \in \U \IMPL Y\in \U\right) \\
  \AND\, &\left(X,Y\in \U \IMPL (X\cap Y)\in\U\right)  \\
  \AND\, &\left(X\in\U \IMPL \left(k>n \AND k\in X\right)\right) \\
  \AND\, &\left(\U(X) =_0 \U(\lambda n . \min(Xn,1))\right)\big)
.\end{aligned}
\]
By coding the sets $X$, $Y$ together into one set $Z$ and calling the quantifier free matrix of the above statement \lp{(\U)_\qf} we arrive at
\[
\Exists{\U^2}\Forall{Z^1}\Forall{n}\Exists{k} \lpp{(\U)_\qf}{\U,Z,n,k}
.\]
Applying \lp[QF]{AC^{1,0}} yields
\begin{equation}\label{eq:ultrand}
  \Exists{\U^2}\Exists{K^2}\Forall{Z^1}\Forall{n} \lpp{(\U)_\qf}{\U,Z,n,KnZ}
.\end{equation}
Note that $\U$ and $K$ are only of degree $2$. This will be crucial for the following proof.


For $K$ one may always choose 
\begin{equation}\label{eq:k'}
  K'(n,X) :=
  \begin{cases}
    \min\{ k \in X \mid k>n\} & \text{if exists,} \\
    0 & \text{otherwise.}
  \end{cases}
\end{equation}
The functional $K'$ is definable using $\mu$.
Therefore the real difficulty lies in finding a solution for $\U$.

We are now in the position to give a proof of Theorem~\ref{thm:progex}.

\begin{proof}[Proof of Theorem~\ref{thm:progex}]
  In the light of Theorem~\ref{thm:ultraca} it is sufficient to prove only that $\ls{RCA_0^\omega} + \lpf{\U}$ is conservative.

  Let $\lf{A_\qf}(f,g)$ be a quantifier-free statement not containing $\U$, such that
  \[
  \ls{RCA_0^\omega} + \lpf{\U} \vdash \Forall{f^1}\Exists{g^1} \lf{A_\qf}(f,g)
  .\]
  By the deduction theorem we obtain 
  \[
  \ls{RCA_0^\omega} \vdash \lpf{\U} \IMPL \Forall{f}\Exists{g} \lf{A_\qf}(f,g)
  .\]
  Using Lemma~\ref{lem:ultraexelim} we obtain
  \[
  \WEPAw + \lp[QF]{AC^{1,0}} \vdash \lpf{\U} \IMPL  \Forall{f}\Exists{g} \lf{A_\qf}(f,g)
  .\]
  Reintroducing a variable $\U$ for the ultrafilter together with \eqref{eq:ultrand} gives   
  \[
  \left(\Exists{\U^2}\Exists{K^2}\Forall{Z^1}\Forall{n} \lpp{(\U)_\qf}{\U,Z,n,KnZ}\right)
  \IMPL \Forall{f}\Exists{g} \lf{A_\qf}(f,g)
  \]
  which is equivalent to
  \[
  \Forall{f}\Forall{\U^2}\Forall{K^2}\Exists{Z^1,n}\Exists{g} \left(\lpp{(\U)_\qf}{\U,Z,n,KnZ} \IMPL \lf{A_\qf}(f,g)\right)
  .\]
  A functional interpretation yields terms $t_Z,t_n,t_g\in T_0[\U,K,f]$ such that
  \begin{equation}\label{eq:ultraimpnd}
    \WEPAw \vdash \Forall{f}\Forall{\U^2}\Forall{K^2} \left(\lpp{(\U)_\qf}{\U,t_Z,t_n,Kt_nt_Z} \IMPL \lf{A_\qf}(f,t_g)\right)
    ,
  \end{equation}
  see for instance Theorem~10.53 in \cite{uK08}.
  Now by Theorem~\ref{thm:cut2} applied to $t_Z,t_n,t_g$ we obtain normalized term $t'_Z,t'_n,t'_g$ which are provably (relative to \WEPAw) equal and such that every occurrence of $\U$ and $K$ is of the form
  \[
  \U(t[j^0]) \qquad\text{resp.}\qquad K(n^0, t[j^0])
  ,\]
  where $t$ is a term in $T_0[\U,K,f]$.

  Let $(t_i)_{i<n}$ be the list of all of these terms $t$ to which $\U$ and $K$ are applied. Assume that this list is partially sorted according to the subterm ordering, i.e.\ if $t_i$ is a subterm of $t_j$ then $i<j$.

  We now build for each $f$ a partial non-principal ultrafilter $\F$ which acts on these occurrences like a real non-principal ultrafilter. For this fix an arbitrary $f$.

  The filter $\F$ is build by iterated applications of Proposition~\ref{pro:ultraex}:\\
  To start the iteration let $\mathcal{A}_{-1}$ be the trivial algebra $\{\emptyset,\Nat\}$ and $\F_{-1}$ be the partial non-principal ultrafilter for $\mathcal{A}_{-1}$.\\
  Let $\mathcal{A}_i$ be the algebra spanned by $\mathcal{A}_{i-1}$ and the sets described by $t_i$ where $\U, K$ are replaced by $\F_{i-1}$ and $K'$ from \eqref{eq:k'}, i.e.\ ${\big(t_i[\U/ \F_{i-1},K/K'] (j)\big)}_{j\in\Nat}$. Let $\F_i$ be an extension of $\F_{i-1}$ to the new algebra $\mathcal{A}_i$ as constructed in Proposition~\ref{pro:ultraex}.

  Obviously in a term $t_i$ the functional $\F$ is only applied to subterms of $t_i$.
  Since the $(t_i)$ is sorted according to the subterm ordering the partial non-principal ultrafilter is already fixed for this applications.

  For the resulting partial non-principal ultrafilter $\F:=\F_n$ we then get
  \[
  \lpp{(\U)_\qf}{\F,t_Z[\F_n,K',f] ,t_n[\F,K',f],K't_n[\F,K',f]t_Z[\F,K',f]}
  .\]
  and in total
  \[
  \ls{RCA_0^\omega} + \lpf{\mu} \vdash \Forall{f} \Exists{\F} \lpp{(\U)_\qf}{\F,t_Z[\F,K',f] ,t_n[\F,K',f],K't_n[\F,K',f]t_Z[\F,K',f]}
  .\]
 
  Combining this with \eqref{eq:ultraimpnd} yields
  \[
  \ls{RCA_0^\omega} + \lpf{\mu} \vdash \Forall{f} \Exists{\F} \lf{A_\qf}(f,t_g[\F,K',f])
  \]
  and hence
  \[
  \ls{RCA_0^\omega} + \lpf{\mu} \vdash \Forall{f} \Exists{g} \lf{A_\qf}(f,g)
  .\]
  With this we have eliminated the use of $\lpf{\U}$ in the proof. 

  By Theorem~8.3.4 of \cite{AF98} the theory $\ls{RCA_0^\omega} + \lpf{\mu}$ is conservative over \ls{ACA_0^\omega} and therefore 
   \[
  \ls{ACA_0^\omega} \vdash \Forall{f} \Exists{g} \lf{A_\qf}(f,g)
  .\]

  To obtain a realizer for $g$ use again the functional interpretation on the last statement. This extracts a realizer $t \in T_0[B_{0,1}]$ where $B_{0,1}$ is the bar recursor of lowest type, see Section~11.3 in \cite{uK08}. Since $t_g$ is only a term of type $2$ one can find a term $t'\in T$ which is equal to $t$, see \cite[Corollary~4.4.(1)]{uK99}. This $t'$ solves the theorem.
\end{proof}

If one is not interested in the extracted program then one can obtain a stronger conservation result:

\begin{theorem}[Conservation]
  The system $\ls{ACA_0^\omega} + \lpf{\mu} + \lpf{\U}$ is $\Pi^1_2$-conservative over $\ls{ACA_0^\omega}$ and therefore also conservative over $\ls{PA}$. 
\end{theorem}
\begin{proof}
  Let $\Forall{f}\Exists{g} \lf{A}(f,g)$ be an arbitrary $\Pi^1_2$ statement which is provable in $\ls{ACA_0^\omega} + \lpf{\mu} + \lpf{\U}$ and does not contain $\mu$ or $\U$. We will show that this statement is provable in $\ls{ACA_0^\omega}$ and if it is arithmetical also in $\ls{PA}$.

  Relative to $\lpf{\mu}$ each arithmetical formula is equivalent to a quantifier free formula. Hence there exists a quantifier free formula $\lf{A_\qf'}$ such that
  \[
  \ls{RCA_0^\omega} + \lpf{\mu} \vdash \lf{A}(f,g) \IFF \lf{A_\qf'}(f,g)
  .\]
  This gives
  \[
  \ls{RCA_0^\omega} + \lpf{\mu} + \lpf{\U} \vdash \Forall{f}\Exists{g} \lf{A_\qf'}(f,g)
  .\]
  Since the system $\ls{RCA_0^\omega} + \lpf{\mu}$ has a functional interpretation in $T_0[\mu]$, see \cite[8.3.1]{AF98}, one can now apply the same argument as in the proof of Theorem~\ref{thm:progex} with $T_0$ is replaced by $T_0[\mu]$, and obtains that
  \[
  \ls{RCA_0^\omega} + \lpf{\mu} \vdash \Forall{f}\Exists{g} \lf{A_\qf'}(f,g)
  \]
  and therefore also 
  \[
  \ls{RCA_0^\omega} + \lpf{\mu} \vdash \Forall{f}\Exists{g} \lf{A}(f,g)
  .\]

  The result follows now also from Theorem~8.3.4 of \cite{AF98}.
\end{proof}

\appendix
\section{Elimination of Skolem functions for monotone formulas}
We will show in this appendix that uses of a partial non-principal ultrafilter for an algebra given by a fixed term over a weak basis theory does not lead to more than primitive recursive growth. For this we will make use of Kohlenbach's elimination of Skolem functions for monotone formulas, see \cite{uK98a}, \cite[Chapter~13]{uK08}.

Let \ls{WKL_0^*} be the system \ls{WKL} where \lp[\Sigma^0_1]{IA} is replaced by \lp[QF]{IA} and the exponential function and let \lp{{WKL_0^\omega}^*} be the corresponding finite type extension. For a formal definition of see \cite[X.4.1]{sS99} and \cite{uK05b} for the finite type system.

Let \lpp[\Pi^0_1]{CA}{f} be the restriction of $\Pi^0_1$-comprehension to the $\Pi^0_1$ formula given by $f$, i.e.\ the statement
\[
\Exists{g} \Forall{n}\left(g(n)=0 \IFF \Forall{x} f(n,x) = 0\right)
\]
Further, let \lpp{\U}{\mathcal{A}} be the principle that states that for the algebra $\mathcal{A}=(A_n)_{n\in\Nat}$ given by $(f(n))_{n\in\Nat}$ there exists a set $F\subseteq \Nat$, such that 
\[
\F = \{ A \mid \Exists{n\in F} (A=A_i)\}
\]
satisfies \lpf{\U} relativized to $\mathcal{A}$.
This means that
\[
\left\{
  \begin{aligned}
    &  \Forall{i,j} \left(A_i = \overline{A_j} \IMPL \left(i\in F \OR j\in F\right)\right) \\
    \AND\, & \Forall{i,j} \left(\left(A_i \subseteq A_j \AND i\in F \right)\IMPL j\in F \right)\\
    \AND\, & \Forall{i,j,k} \left(\left(i,j\in F \AND A_k = A_i\cap A_j \right) \IMPL k\in F\right) \\
    \AND\, & \Forall{i} \left(i\in F \IMPL \Forall{n}\Exists{k>n} (k\in A_i)\right) .
  \end{aligned}
\right.
\]

We obtain the following theorem:
\begin{theorem}
  Let $\lf{A_\qf}(f,x)$ be a quantifier free formula that contains only $f,x$ free and let $t_1,t_2$ be terms in \ls{{WKL_0^\omega}^*}.
  If 
  \[
  \ls{{WKL_0^\omega}^*} \vdash
  \Forall{f} \left(\lpp[\Pi^0_1]{CA}{t_1f} \AND \lpp{\U}{t_2f} \IMPL \Exists{x} \lf{A_\qf}(f,x)
  \right)
  \]
  then one can extract a primitive recursive (in the sense of Kleene) functional $\Phi$ such that
  \[
  \ls{RCA_0^\omega} \vdash \Forall{f} \lf{A_\qf}(f,\Phi(f))
  .\]

  In particular if $f$ is only of type $0$ one obtains that there exists a primitive recursive function $g$ such that
  \[
  \ls{PRA} \vdash \Forall{x} \lf{A_\qf}(x,g(x))
  .\]
\end{theorem}
\begin{proof}
  We will show, by formalizing the construction of $b$ in the proof of Proposition~\ref{pro:ultraex}, that there exists a term $t'$ such that
  \begin{equation}\label{eq:t'}\notag
  \Forall{h} \left(\lpp[\Pi^0_1]{CA}{t'h} \IMPL \lpp{\U}{h}\right)
  .
  \end{equation}

  The theorem follows then from the elimination of Skolem functions for monotone formulas
  and the fact that one can code the two instances of \lp[\Pi^0_1]{CA} given by $t_1$ and $t't_2$ into one. For the elimination of Skolem functions see  for instance Proposition~13.20 in \cite{uK08} --- the statement of this proposition is essentially the same as of this theorem without $\U$. For the conservativity over \ls{PRA}, see \cite{AF98}.

  In the construction of $b$ in the proof of Proposition~\ref{pro:ultraex} only two steps cannot be formalized in \ls{{WKL_0^\omega}^*}.
  The first step is the application of \lp[\Pi^0_1]{CP} and the second is the use of \lp[\Pi^0_2]{WKL}.
  The use of \lp[\Pi^0_1]{CP} can be reduced to a suitable instance of \lp[\Pi^0_1]{CA} (with the parameters $\F,\mathcal{\tilde{A}}$) and \lp[QF]{AC^{1,0}}. The use of \lp[\Pi^0_2]{WKL} follows from \lp[\Pi^0_1]{WKL} and another instance of \lp[\Pi^0_1]{CA} (also with the parameters $\F,\mathcal{\tilde{A}}$).
  Since \lp[\Pi^0_1]{WKL} is equivalent to \lp{WKL} and one can code the two instances of comprehension together one obtains in total that the index function $b$ can be constructed in $\lp{{WKL_0^\omega}^*} + \lpp[\Pi^0_1]{CA}{t\F\mathcal{\tilde{A}}}$ for a suitable $t$. (Note that the set $\F$ cannot be defined since it involves $\mu$.)

  Using this one can extend the partial ultrafilter $\F=\{\Nat\}$ on the trivial algebra $\mathcal{A}=\{\emptyset,\Nat\}$ to an (index set of an) ultrafilter satisfying \lpp{\U}{h}. From this one can easily construct a term $t'$. This provides the theorem.
\end{proof}

\begin{remark}
  Although the restriction of $\U$ to an algebra given by a term seems to be weak, it is strong enough to prove instances of ultralimit, i.e.\ that the ultralimit exists for (a sequence of) sequences given by one fixed term.

  To see this let $(x_n)_{n\in \Nat}$ be a sequence in the interval $[0,1]$. We will prove that the ultralimit of this sequence exists using $\lpp{(\U)}{t[(x_n)]}$ for a term $t$.
  For this let
  \[
  A_{i,k} := \left\{\, n\in\Nat \sizeMid x_n\in \left[\frac{i}{2^k}, \frac{i+1}{2^k}\right[ \,\right\}
  .\]
  Let $\mathcal{A}$ be the algebra created by this sets. It is clear that $\mathcal{A}$ can be described by a term $t[(x_n)]$.
  
  Observed that the proof of Lemma~\ref{lem:finitepart} can also be carried out in \ls{RCA_0^*}.
  Since $\big(A_{i,k}\big)_{i\le 2^k}$ defines a finite partition of $\Nat$, Lemma~\ref{lem:finitepart} provides
  \[
  \Forall{k}\ExistsUn{i\le 2^k} \left(A_{i,k}\in \U\right)
  ,\]
  (strictly speaking we obtain that the index of $A_{i,k}$ is in an index set of $\U$)
  and \lp[QF]{AC^{1,0}} yields a choice function $f(k)$ for $i$.
  Note that the ultrafilter properties provide that each $A_{f(k),k}$ is infinite and that
  \[
  \Forall{k}\Forall{k'>k} \left(A_{f(k'),k'}\subseteq A_{f(k),k}\right)
  .\]

  Let $g(k)$ be the $k$\nobreakdash-th element of $A_{f(k),k}$ then
  the sequence $\big(x_{g(k)}\big)_k$ defines a Cauchy-sequence with Cauchy-rate $2^{-k}$ which converges to $\lim_{n\to\U} x_n$. 
\end{remark}

\bibliographystyle{amsplain}
\bibliography{swkl,primrec,ultra}

\providecommand{\bysame}{\leavevmode\hbox to3em{\hrulefill}\thinspace}
\providecommand{\MR}{\relax\ifhmode\unskip\space\fi MR }
\providecommand{\MRhref}[2]{%
  \href{http://www.ams.org/mathscinet-getitem?mr=#1}{#2}
}
\providecommand{\href}[2]{#2}
\begin{thebibliography}{10}

\bibitem{AK90}
Asuman~G. Aksoy and Mohamed~A. Khamsi, \emph{Nonstandard methods in fixed point
  theory}, Universitext, Springer-Verlag, New York, 1990, With an introduction
  by W. A. Kirk. \MR{1066202}

\bibitem{jA98}
Jeremy Avigad, \emph{An effective proof that open sets are {R}amsey}, Arch.
  Math. Logic \textbf{37} (1998), no.~4, 235--240. \MR{1635557}

\bibitem{jA05}
\bysame, \emph{Weak theories of nonstandard arithmetic and analysis}, Reverse
  mathematics 2001, Lect. Notes Log., vol.~21, Assoc. Symbol. Logic, La Jolla,
  CA, 2005, pp.~19--46. \MR{2185426}

\bibitem{AF98}
Jeremy Avigad and Solomon Feferman, \emph{G\"odel's functional
  (``{D}ialectica'') interpretation}, Handbook of proof theory, Stud. Logic
  Found. Math., vol. 137, North-Holland, Amsterdam, 1998, pp.~337--405.
  \MR{1640329}

\bibitem{BBS}
Benno van~den Berg, Eyvind Briseid, and Pavol Safarik, \emph{{NN}}, in
  preparation.

\bibitem{pG06b}
Philipp Gerhardy, \emph{A quantitative version of {K}irk's fixed point theorem
  for asymptotic contractions}, J. Math. Anal. Appl. \textbf{316} (2006),
  no.~1, 339--345. \MR{2201765}

\bibitem{jH04}
Jeffry~L. Hirst, \emph{Hindman's theorem, ultrafilters, and reverse
  mathematics}, J. Symbolic Logic \textbf{69} (2004), no.~1, 65--72.
  \MR{2039345}

\bibitem{jK06}
H.~Jerome Keisler, \emph{Nonstandard arithmetic and reverse mathematics}, Bull.
  Symbolic Logic \textbf{12} (2006), no.~1, 100--125. \MR{2209331}

\bibitem{KS01}
Mohamed~A. Khamsi and Brailey Sims, \emph{Ultra-methods in metric fixed point
  theory}, Handbook of metric fixed point theory, Kluwer Acad. Publ.,
  Dordrecht, 2001, pp.~177--199. \MR{1904277}

\bibitem{uK98a}
Ulrich Kohlenbach, \emph{Elimination of {S}kolem functions for monotone
  formulas in analysis}, Arch. Math. Logic \textbf{37} (1998), 363--390.
  \MR{1634279}

\bibitem{uK99}
\bysame, \emph{On the no-counterexample interpretation}, J. Symbolic Logic
  \textbf{64} (1999), no.~4, 1491--1511. \MR{1780065}

\bibitem{uK05b}
\bysame, \emph{Higher order reverse mathematics}, Reverse mathematics 2001,
  Lect. Notes Log., vol.~21, Assoc. Symbol. Logic, La Jolla, CA, 2005,
  pp.~281--295. \MR{2185441}

\bibitem{uK08}
\bysame, \emph{Applied proof theory: Proof interpretations and their use in
  mathematics}, Springer Monographs in Mathematics, Springer Verlag, 2008.
  \MR{2445721}

\bibitem{swkl}
Alexander~P. Kreuzer and Ulrich Kohlenbach, \emph{Term extraction and
  {R}amsey's theorem for pairs}, submitted, preprint available at
  \url{http://www.mathematik.tu-darmstadt.de/~akreuzer/files/TermExtractionAnd%
RT22.rev.pdf}.

\bibitem{eP99}
Erik Palmgren, \emph{An effective conservation result for nonstandard
  arithmetic}, Math. Log. Q. \textbf{46} (2000), no.~1, 17--23. \MR{1736646}

\bibitem{sS99}
Stephen~G. Simpson, \emph{Subsystems of second order arithmetic}, Perspectives
  in Mathematical Logic, Springer-Verlag, Berlin, 1999. \MR{1723993}

\bibitem{rS78}
Robert~M. Solovay, \emph{Hyperarithmetically encodable sets}, Trans. Amer.
  Math. Soc. \textbf{239} (1978), 99--122. \MR{0491103}

\bibitem{hT11}
Henry Towsner, \emph{Hindman's theorem: an ultrafilter argument in second order
  arithmetic}, J. Symbolic Logic \textbf{76} (2011), no.~1, 353--360.

\end{thebibliography}

\end{document}